\documentclass[a4paper,notitlepage,twoside,leqno,12pt]{amsart}

\usepackage{bbm,pifont,latexsym,amssymb}
\usepackage{anysize}
\marginsize{3.8cm}{3.8cm}{3.8cm}{3.8cm}

\usepackage{dcolumn,indentfirst}
\usepackage[colorlinks=false,linkcolor=black,citecolor=black,urlcolor=black,bookmarksopen=true,linktocpage=true,pdfstartview=FitH]{hyperref}
\usepackage{amsmath,amssymb,amscd,amsthm,amsfonts,mathrsfs}
\usepackage{color,graphicx,xcolor,graphics,subfigure,extarrows,caption2}
\usepackage{tikz,tikz-cd}
\usepackage{titletoc}
\usepackage{xy}

\newtheorem{thm}{Theorem}[section]
\newtheorem*{main}{Main Theorem}
\newtheorem{cor}[thm]{Corollary}
\newtheorem{prop}[thm]{Proposition}
\newtheorem{lem}[thm]{Lemma}

\theoremstyle{definition}
\newtheorem{defi}{Definition}

\newtheorem{exam}{Example}

\newcommand{\D}{\mathbb{D}}

\newcommand{\R}{\mathbb{R}}
\newcommand{\Z}{\mathbb{Z}}
\newcommand{\Q}{\mathbb{Q}}
\newcommand{\C}{\mathbb{C}}
\newcommand{\N}{\mathbb{N}}
\newcommand{\EC}{\widehat{\mathbb{C}}}

\newcommand{\MP}{\mathcal{P}}

\newcommand{\MC}{\mathcal{C}}

\newcommand{\diam}{\textup{diam}}
\newcommand{\dist}{\textup{dist}}

\newcommand{\ii}{\textup{i}}

\usepackage{enumerate,enumitem}

\makeatletter\@addtoreset{equation}{section}\makeatother

\begin{document}

\author[Y. Fu]{Yuming Fu}
\address{School of Mathematical Sciences, Shenzhen University, Shenzhen 518060, P. R. China}
\email{yumingfuxy@szu.edu.cn}

\author[Y. Zhang]{Yanhua Zhang*}
\address{School of Mathematical Sciences, Qufu Normal University, Qufu 273165, P. R. China}
\email{zhangyh0714@qfnu.edu.cn}

\title[Mating Siegel and Thurston quadratic polynomials]{Mating Siegel and Thurston quadratic polynomials}

\begin{abstract}
We prove that a quadratic polynomial with a bounded type Siegel disk and a quadratic postcritically finite polynomial are always mateable.
\end{abstract}



\date{\today}

\maketitle

\section{Introduction}
\medskip

\subsection{Siegel quadratic polynomials and Thurston quadratic polynomials}
Let $0<\theta<1$ be a bounded type irrational number. Any quadratic polynomial having a fixed Siegel disk with rotation number $\theta$ is conjugate to
\begin{equation*}
    P_{\theta}:z\mapsto e^{2\pi i \theta}z+z^2
\end{equation*}
by an affine map. According to Douady-Herman (see \cite{Dou87}, \cite{Her87}), $P_\theta$ has a Siegel disk $D_{P_\theta}$ centered at zero whose boundary is a quasi-circle containing the critical point $x_\theta=-e^{2\pi i \theta}/2$ of $P_\theta$. Petersen proved that the Julia set $J(P_\theta)$ is locally connected and has zero Lebesgue measure in \cite{Pet96} (see also \cite{Yam99}). Note that every bounded Fatou component of $P_\theta$ is a Jordan disk, which is a preimage of $D_{P_\theta}$.

Given a polynomial $P$ with degree at least two, let us denote by $Crit(P)$ the set of all critical points of $P$ in $\C$. The \textit{postcritical set} of $P$ is
$$\mathcal{P}(P):=\overline{\bigcup_{n\geq 1 } P^n(Crit(P))}.$$
We say that $P$ is postcritically finite if $\mathcal{P}(P)$ is a finite set. This kind of polynomial is also called Thurston's type.  Let 
\begin{equation*}
    P_{c}:z\mapsto z^2+c
\end{equation*}
be a Thurston's type quadratic polynomial, i.e., the orbit of zero is finite. Then $J(P_c)$ is connected by \cite{Fat19}. When the critical point $x_0=0$ of $P_c$ is preperiodic but not periodic, the Julia set $J(P_c)$ is a dendrite, which is a compact, pathwise connected, locally connected and nowhere dense set that does not separate the plane (see \cite[Theorem 4.2]{CG93}). When $x_0$ is periodic, $P_c$ has a super-attracting cycle. Then $J(P_c)$ is locally connected by \cite[Theorem 4.1]{CG93}. Furthermore, every bounded Fatou component of $P_c$ is a Jordan disk (see \cite{RY22}) which is one of the preimages of this cycle. 

\subsection{Mating: Definitions and the main result}
Mating two polynomials with the same degree into a rational map is an idea introduced by Douady and Hubbard in \cite{Dou83} to understand the rational dynamics. Let $\C$, $\EC$, $\Delta$ denote the complex plane, the Riemann sphere, and the open unit disk, respectively. Suppose $P_1$ and $P_2$ are two polynomials of degree $d \ge 2$ so that the Julia sets $J(P_1)$ and $J(P_2)$ are connected and locally connected. Let $\Phi_{P_j}$, $j=1,2$, be the B\"{o}ttcher map of $P_j$ which maps the basin of infinity $\EC\setminus K(P_j)$ onto the outside of the closed unit disk $\EC\setminus \overline{\D}$ such that
$$\Phi_{P_j}(P_j(z))=(\Phi_{P_j}(z))^2\quad \textrm{for all}\quad z\in\EC\setminus K(P_j),$$
where $K(P_j)$ is the filled Julia set of $P_j$. We use the notation 
$$R_t(P_j):=\Phi^{-1}_{P_j}(\{re^{2\pi i t}|1< r< \infty\}) $$ 
for the external ray of $P_j$ with angle $t\in\R/\Z$, where $j=1,2$. 
By Carath\'{e}odory's theorem, the inverse map $\Phi^{-1}_{P_j}$ has a continuous extension to their corresponding boundaries. Then the limit
$$x=\underset{r \to 1^+}{\lim}\Phi^{-1}_{P_j}(re^{2\pi i t})$$ 
exists, and we say that $R_t(P_j)$ lands at the point $x$.

The general idea of the mating can be described as follows. First, we add a circle at infinity to $\C$ and get 
$$\tilde{\C} = \C \cup \{\infty \cdot e^{2 \pi i t}| t\in\R/\Z\}.$$
With the natural topology, $\tilde{\C}$ is homeomorphic to the closed unit disk. Then we extend 
$$P_1: \C \to \C   \hbox{ and }   P_2: \C \to \C$$
continuously to
$$\tilde{P_1}: \tilde{\C}  \to \tilde{\C}   \hbox{ and } \tilde{P_2}: \tilde{\C}  \to \tilde{\C}$$
so that 
$$\tilde{P_1}(\infty\cdot e^{2 \pi i t})  = \tilde{P_2}(\infty\cdot e^{2 \pi i t})   =  \infty \cdot e^{2 \pi i dt}.$$ 
To distinguish, we denote $\tilde{\C}$ as $\tilde{\C}_{P_1}$ and $\tilde{\C}_{P_2}$ respectively for $P_1$ and $P_2$. By gluing $\tilde{\C}_{P_1}$ and $\tilde{\C}_{P_2}$ along the circle at  infinity so that  $\infty \cdot e^{2 \pi i t}$ and $\infty \cdot e^{2 \pi i (1-t)}$ are identified, we get a two-sphere $S^2 = \tilde{\C}_{P_1} \sqcup\tilde{\C}_{P_2}/\sim$.  Define
$$P_1 \sqcup P_2:  S^2\to S^2$$ 
by setting $P_1 \sqcup P_2|_{\tilde{\C}_{P_1}}= P_1$ and $P_1 \sqcup P_2|_{\tilde{\C}_{P_2}} = P_2$. We say that $P_1 \sqcup P_2$ is the formal mating of $P_1$ and $P_2$. Since we added the circle to $\mathbb{C}$, the conformal structure on $S^2$ has been destroyed. However, we can keep the conformal structure in a neighborhood of the filled Julia sets $K(P_1)$ and $K(P_2)$.

In the case that $P_1$ and $P_2$ are Thurston's type polynomials, the map $P_1\sqcup P_2$ must be a postcritically finite branched covering of $S^2$ to itself which preserves the orientation. Suppose $P_1\sqcup P_2$ has no Thurston obstructions by removing trivial obstructions if necessary. By Thurston's Theorem (\cite{DH93}), there is a rational map which is combinatorially equivalent to $P_1\sqcup P_2$ rel $\mathcal{P}(P_1\sqcup P_2)$. In this case, we say that $P_1$ and $P_2$ are strongly mateable.

There is another viewpoint on the mating which applies to the case of postcritically infinite polynomials as well. To get that, let us see how $\tilde{\C}_{P_1}\sqcup\tilde{\C}_{P_2}$ is deformed as we iterate the Thurston's pullback map (\cite{BEKP09}). We may view $\tilde{\C}_{P_1} \sqcup\tilde{\C}_{P_2}$ as the union of two filled Julia sets $K(P_1)$, $K(P_2)$ and the family of joined external rays 
\begin{equation*}
    R_t = \overline{R_t(P_1)}\sqcup \overline{R_{1-t}(P_2)}, t\in\R/\Z,
\end{equation*}
where 
$$\overline{R_t(P_j)}=R_t(P_j)\cup\{\Phi^{-1}_{P_j}(e^{2\pi it})\}\cup\{\infty\cdot e^{2\pi it}\in\tilde{\C}_{P_j}\},~~j=1,2.$$
\begin{defi}
    We say two points $x\in J(P_1)$ and $y\in J(P_2)$ are \emph{ray equivalent} if and only if $x$ and $y$ are connected by a path consisting of joined external rays.
\end{defi}

Suppose $P_1\sqcup P_2$ is combinatorially equivalent to some rational map $R$.  Let 
$$\phi_1, \phi_2:  \tilde{\C}_{P_1} \sqcup\tilde{\C}_{P_2} \to \EC$$ 
be homeomorphisms which are isotopic to each other rel $\mathcal{P}(P_1\sqcup P_2)$ such that
\begin{equation*}
    R = \phi_1 \circ P_1\sqcup P_2 \circ \phi_2^{-1}.
\end{equation*}
Let $[\phi_1]$ be the isotopy class of $[\phi_1]$ rel $\mathcal{P}(P_1\sqcup P_2)$.
Since $[\phi_1] = [\phi_2]$ is the fixed point of the Thurston's pullback map, we may get a sequence of homeomorphisms $\phi_n: \tilde{\C}_{P_1} \sqcup\tilde{\C}_{P_2} \to \EC$ with $[\phi_n] = [\phi_1]$ for all $n \ge 1$ such that 
\begin{equation*}
    R=\phi_n \circ P_1\sqcup P_2 \circ \phi_{n+1}^{-1}.
\end{equation*}
For each $n \ge 1$, $\phi_n(K(P_1))$ and $\phi_n(K(P_2))$ are two homeomorphic copies of $K(P_1)$ and $K(P_2)$ which are disjoint from each other, and they are connected by the family of rays $\phi_n(R_t)$, $t\in\R/\Z$. In many known cases, in particular for the postcritically finite case, we have    
\begin{equation} \label{slow-mating}
 {\rm diam}(\phi_n(R_t)) \to 0  \hbox{  as  } n\to\infty,
\end{equation} 
where $\rm{diam}(\cdot)$ denotes the diameter with respect to the spherical metric. We call (\ref{slow-mating}) \emph{the slow-mating property}.

Now suppose (\ref{slow-mating}) holds. Then as $n \to \infty$, the two homeomorphic copies of the filled Julia sets become closer and closer to each other and finally glue into the whole sphere, the $\phi_n$-images of the points in the Julia sets $J(P_1)$ and $J(P_2)$ which are connected by rays collide into a single point and $\phi_n$ converges uniformly to a continuous map $\phi$. The action of the rational map on $\phi(K(P_1))$(resp. $\phi(K(P_2))$) is semi-conjugate to $K(P_1)$(resp. $K(P_2)$). Let us summarize these as

\begin{defi}\label{mating}
We call a rational map $R$ the mating of $P_1$ and $P_2$ if $P_1\sqcup P_2$ is combinatorially equivalent to $R$ so that the slow-mating property (\ref{slow-mating}) is satisfied, and moreover, there is a continuous semi-conjugacy $\phi: \EC  \to \EC$ with   
$$\phi \circ (P_1 \sqcup P_2)  = R \circ \phi$$
which is conformal in the interior of $K(P_1)$ and $K(P_2)$, such that
$$\phi (K(P_1)) \cup \phi(K(P_2))  = \EC,$$ 
and furthermore, $\phi(x) = \phi(y)$ if and only if $x$ and $y$ are ray equivalent to each other.
\end{defi}

There are several types of mating. We refer the reader to \cite{PM12} for a thorough explanation of their definitions.

The main result of this paper is

\begin{main}
Let $P_\theta$ be a quadratic polynomial with a bounded type Siegel disk and $P_c$ be a quadratic postcritically finite polynomial. Then $P_\theta$ and $P_c$ are mateable in the sense of Definition \ref{mating}.
\end{main}

The proof of the main theorem is based on two tools. The first one is a characterization theorem of certain class of rational maps with Siegel disks in \cite{Zha22}, which allows us to find a candidate rational map $G$ for the mating of $P_\theta$ and $P_c$. The second one is a contraction lemma in \cite{WYZZ23} which allows us to prove that $G$ does realize the mating of $P_\theta$ and $P_c$.

\subsection{Statement of some known results}
For the Main Theorem, a particular case $P_c(z)=z^2-2$ which is known as \emph{Chebyshev} quadratic polynomial appeared as a special example in \cite{YZ01}. Another particular case in \cite{Yan15} is $P_c(z)=z^2-1$ with a super-attracting cycle of period two. It is worth mentioning that the matability of the Basilica (i.e., the Julia set of \( z \mapsto z^2 - 1 \)) with some other quadratic polynomials, such as non-renormalizable quadratic polynomials, etc (see \cite{Tim08}, \cite{AY09}, and \cite{Dud11}).

In general, to solve the mating problem of two quadratic polynomials, it is necessary to first find the candidate rational map and then to prove that this model function is actually the desired mating (see \cite{Pet99}, \cite{YZ01}, \cite{Yan15} and \cite{FY23} for examples). In this paper, we will prove our main theorem using a different argument from \cite{YZ01} and \cite{Yan15}. Specifically, we do not rely on analytic Blaschke models and complex a prior bounds.

In some certain cases, it is difficult to obtain the explicit formulas of rational maps after mating. Rees proved if the formal mating of two hyperbolic polynomials $P_1$ and $P_2$ of Thurston's type is combinatorially equivalent to a rational map, then the topological mating is conjugate to the rational map (see \cite{Ree92}). In other words, if there is no Thurston obstructions for $P_1\sqcup P_2$, then the rational map $R$ which is combinatorially equivalent to $P_1\sqcup P_2$ can be realized the mating of $P_1$ and $P_2$ in the sense of Definition \ref{mating}. Shishikura proved that this is also true even if the polynomials are merely Thurston's type and the formal mating can be replaced by the degenerate mating (see \cite{Shi00}). 

Besides, there are some efficient algorithms which can be used to compute numerical approximations of the explicit formulas of rational maps after mating, see \cite{BH12}, \cite{Jun17} and \cite{Wil19}. 

Rees, Shishikura and Tan proved that any two pair of subhyperbolic quadratic polynomials are conformally mateable if they are not in conjugate limbs of the Mandelbrot set (\cite{Ree86}, \cite{Ree92}, \cite{Shi00}, \cite{Tan92}). The mateable results on subhyperbolic quadratics have been extended to geometrically finite case by parabolic surgery (\cite{HT04}). Because of the emergence of more critical orbits and more combinatorics, there are some essential differences between the mating of higher degree polynomials (see \cite{Che12}, \cite{AR16} and \cite{Sha19} for the mating of special cubic polynomials). Furthermore, Shishikura and Tan have showed that there exist cubic polynomials which are topologically mateable, but the resulting branched covering fails to be equivalent to a rational map in \cite{ST00}. Restricting attention to $S_1$ which is a space of monic, centered cubic polynomials with a marked fixed critical point, Sharland showed that there exist two postcritically finite polynomials in $S_1$ which do not lie in conjugate limbs of $S_1$, but they are not mateable in \cite{Sha23}. For more questions on mating polynomials, see \cite{BEKMPRT12}.

\textbf{Acknowledgments.} The authors would like very much to thank Professor GaoFei Zhang for his introductions, support and suggestions. Without his guidance,  this work would not have existed. Thanks are due to Professors Guizhen Cui, Fei Yang and Dr. Luxian Yang, for their encouragement and fruitful discussions with authors. In addition, the research is supported by National Natural Science Foundation of China (Grant Number: \textbf{11801305, 12171276 and 12371085}) and the Natural Science Foundation of Shandong Province (Grant Number: \textbf{ZR2023MA044}).

\section{The candidate rational map of the mating}

Let $0 < \theta < 1$ be an irrational number of bounded type and be fixed throughout. Let $R^{geom}_{\theta}$ be the class of all rational maps $g$ such that:
\begin{enumerate}
    \item $g$ has a fixed Siegel disk $D_g$ with rotation number $\theta$;
    \item the forward orbit of each critical point of $g$ either intersects $\overline{D_g}$, or is eventually periodic, or belongs to the basin of some attracting cycle.
\end{enumerate}
The object of this section is to find the candidate rational map $G\in R^{geom}_{\theta}$ of the mating $P_\theta\sqcup P_c$ by the main result in \cite{Zha22}.

Let $R_\theta^{top}$ denote the class of all orientation-preserving branched coverings $f$ of the sphere to itself that have finite degree and satisfy the following properties: 
\begin{itemize}
    \item the restriction of $f$ in the unit disk $\Delta$ is the rigid rotation $\rho_\theta(z)=e^{2\pi i\theta}z$;
    \item $f$ has at least one critical point on the unit circle;
    \item the forward orbit of every critical point is either eventually finite, attracted to some holomorphic attracting cycle, or intersects the closed unit disk.
\end{itemize}

\begin{defi}\label{com-equivalent}(\cite{Zha22})\label{zhang}
    We say a map $f\in R_\theta^{top}$ is combinatorially equivalent to a map $g\in R^{geom}_{\theta}$ if there exist a pair of homeomorphisms $\psi_1$, $\psi_2:\EC\to \EC$ such that:
    \begin{enumerate}
        \item $\psi_1\circ f= g\circ \psi_2$;
        \item $\psi_1|_{\Delta}= \psi_2|_{\Delta}$ is holomorphic;
        \item for each point $x_i$ in the holomorphic attracting cycles of $f$, there is a Jordan disk $D_i$ containing $x_i$ such that $\psi_1|_{D_i}= \psi_2|_{D_i}$ is holomorphic, and $\psi_1$ is isotopic to $\psi_2$ rel $\mathcal{P}(f)\cup \cup_i\overline{D_i}$.
    \end{enumerate}
\end{defi}

\begin{defi}
Suppose $f:\EC\to \EC$ is an orientation-preserving branched covering map. We say that a simple closed curve $\gamma\in \EC\setminus \mathcal{P}(f)$ is \emph{non-peripheral} if each component of $\EC\setminus\gamma$ contains at least two points in $\mathcal{P}(f)$. Let $\Gamma=\{\gamma,\cdots,\gamma_n\}$ be a family of disjoint non-peripheral curves which are not homotopic to each other. We call such $\Gamma$ a \emph{multi-curve}. We say that a multi-curve $\Gamma$ is \emph{stable} if for each $\gamma_j\in\Gamma$, all the non-peripheral components of $f^{-1}(\gamma_j)$ are homotopic to the elements in $\Gamma$. Associated to each stable multi-curve $\Gamma$, there is a linear transformation $f_{\Gamma}:\R^{\Gamma}\to \R^{\Gamma}$ with non-negative entries: let $\gamma_{i,j,\alpha}$ denote all the non-peripheral components of $f^{-1}(\gamma_j)$ homotopic to $\gamma_i$, and let $d_{i,j,\alpha}$ denote the covering degree of $f:\gamma_{i,j,\alpha}\to \gamma_j$, define
\begin{equation*}
    f_{\Gamma}(\gamma_j)=\sum_{i}a_{i,j}\gamma_i,
\end{equation*}
where $a_{i,j}=\sum\limits_{\alpha}1/d_{i,j,\alpha}$. Let $\lambda_{\Gamma}$ denote the maximal eigenvalue of $f_{\Gamma}$. We call $\Gamma$ a \emph{Thurston obstruction} if $\lambda_{\Gamma}\ge 1$.
\end{defi}

Every $g\in R^{geom}_{\theta}$ is modeled by some $f\in R_\theta^{top}$ in the sense of Definition \ref{com-equivalent}, since a Siegel disk of a rational map with bounded type rotation number is a quasi-disk with at least one critical point on the boundary by \cite{Zha11}.

\begin{thm}[see \cite{Zha22}]\label{zhang}
Suppose $f \in R_\theta^{top}$. Then $f$ is combinatorially equivalent to some rational map $g\in R_\theta^{geom}$ if and only if $f$ has no Thurston obstructions in $\EC\setminus \overline{\Delta}$, and moreover, if it exists, $g$ must be unique up to M$\ddot{o}$bius conjugation. 
\end{thm}

It is clear that $$F:=P_{\theta}\sqcup P_c$$ is a degree two topological branched covering of the sphere which is modeled by a map in $R_\theta^{top}$ by definition. If $F$ has no Thurston obstructions, by Theorem \ref{zhang}, we can find a quadratic rational map $G$ with a bounded type Siegel disk  which is combinatorially equivalent to $F$. The equivalence maps are conformal when restricted on the Siegel disk of $P_{\theta}$, and if $P_c$ has a super-attracting cycle, they are also conformal in open neighborhoods of the points in this cycle. 

\begin{lem}\label{thurston-equi}
 There exists a quadratic rational map $G \in R_\theta^{geom}$ and a pair of homeomorphisms $\phi_1$, $\phi_2:\EC\to \EC$ such that:
   \begin{enumerate}
        \item $\phi_1\circ F= G \circ \phi_2$;
        \item $\phi_1|_{D_{P_{\theta}}}= \phi_2|_{D_{P_{\theta}}}$ is holomorphic;
        \item for each point $x_i$ in the holomorphic attracting cycles of $F$, there is a Jordan disk $D_i$ containing $x_i$ such that $\phi_1|_{D_i}= \phi_2|_{D_i}$ is holomorphic, and $\phi_1$ is isotopic to $\phi_2$ rel $\mathcal{P}(F)\cup \cup_i\overline{D_i}$.
    \end{enumerate}
\end{lem}

\begin{proof}
For the Siegel disk $D_{P_\theta}$, there exists a conformal map $h:D_{P_\theta}\to \Delta$ such that $h\circ P_\theta\circ h^{-1}(z)=\rho_\theta(z)=e^{2\pi i\theta}z$. Since $\partial D_{P_\theta}$ is a quasi-disk, $h$ has a quasi-conformal extension $H:\EC\to \EC$. Define $\hat{P}_\theta:= H\circ P_\theta\circ H^{-1}$. Then $\hat{P}_\theta$ is quasi-conformally conjugate to $P_\theta$ and $\hat{P}_\theta$ maps $\Delta$ to itself.

Consider the formal mating of $\hat{P}_\theta$ and $P_c$ with a slight difference from $F$, which glues $\tilde{\C}_{\hat{P}_\theta}$ and $\tilde{\C}_{P_c}$ along the circle at the infinity so that $H(\infty\cdot e^{2\pi i t})$ and $\infty\cdot e^{2\pi i (1-t)}$ are identified. Therefore, $\hat{P}_\theta\sqcup P_c\in R_\theta^{top}$ and quasi-conformally conjugates to $F$ by a quasi-conformal map 
\begin{equation}
\hat{H}(z) :=
\begin{cases}
H, &  z\in \tilde{\C}_{\hat{P}_\theta}, \\
id, &  z\in \tilde{\C}_{P_c},
\end{cases}
\end{equation}
where $\hat{H}$ is conformal on $D_{P_\theta}$ and $\tilde{\C}_{P_c}$.

Suppose $\hat{P}_\theta\sqcup P_c$ has a Thurston obstruction $\Gamma$. If a non-peripheral curve $\gamma\in \Gamma$ separates $\overline{D_{P_\theta}}$ and $\mathcal{P}(P_c)$, $\gamma$ cannot be a Thurston obstruction since $(\hat{P}_\theta\sqcup P_c)^{-1}(\gamma)$ is homotopic to $\gamma$ with $\lambda_\Gamma=\frac{1}{2}$. If a complementary component of a non-peripheral curve $\gamma\in \Gamma$ encloses $\overline{D_{P_\theta}}$ and at least one point of $\mathcal{P}(P_c)$, then $\Gamma$ will be a Thurston obstruction of $P_c$ by regarding $\overline{D_{P_\theta}}$ as one point of $\mathcal{P}(P_c)$ located at $\infty$ of $P_c$ and this is impossible. Therefore, $\hat{P}_\theta\sqcup P_c$ has no Thurston obstructions. By Theorem \ref{zhang}, $\hat{P}_\theta\sqcup P_c$ is combinatorially equivalent to some rational map $G\in R_\theta^{geom}$. The proof is complete by setting $\phi_1:=\psi_1\circ\hat{H}$ and $\phi_2:=\psi_2\circ\hat{H}$.
\end{proof}

Following Lemma \ref{thurston-equi}, we still say $F=P_\theta\sqcup P_c$ is combinatorially equivalent to $G$ if the conditions in Lemma \ref{thurston-equi} hold.

\section{The construction of the continuous semi-conjugacy}

Let $f:\EC\to\EC$ be a rational map of degree at least 2. A sequence $\{V_n\}_{n\ge 0}$ is called a pullback sequence of a Jordan disk $V_0$ under $f$,
if $V_{n+1}$ is a connected component of $f^{-1}(V_n)$ for all $n\ge 0$. Let $\diam_{\EC}(\cdot)$ and $\dist_{\EC}(\cdot,\cdot)$ denote respectively the diameter and distance with respect to the spherical metric. 

In this section, we will show the existence of the continuous semi-conjugacy from the formal mating $F$ to the rational map $G$. It suffices to show the uniform convergence of the homotopy lifting of the equivalence maps using the expanding property of $G$ in the following two lemmas:

\begin{lem}[Classical shrinking lemma. \cite{LM97} (see also \cite{Man93}, \cite{TY96})]\label{shrinking lemma}
Let $d\ge1$ and $U_0$, $V_0$ be two Jordan disks in $\EC$. Suppose $U_0$ is not contained in any rotation domain of $f$ and $V_0$ is compactly contained
in $U_0$. Then for any $\epsilon>0$, there exists an $N\ge 1$ such that for any pullback sequence $\{U_n\}_{n\ge 0}$ satisfying 
$\deg (f^{n}:U_n\to U_0)\le d$, $\diam_{\EC}(V_n)<\epsilon$
for all $n\ge N$, where $V_n$ is any component of $f^{-n}(V_0)$ contained in $U_n$.
\end{lem}

\begin{lem}[see the Main Lemma of \cite{WYZZ23}]\label{contraction}
Let $f$ be a rational map with degree at least two and having a fixed bounded type Siegel disk $D_f$.  Suppose $\dist_{\EC}(\MP(f)\setminus\partial{D_f},\partial{D_f})>0$.  Then for any $\epsilon>0$ and any Jordan domain $V_0\subset \EC\setminus\overline{D_f}$ with $\overline{V}_0 \cap \MP(f)\neq\emptyset$ and $\overline{V}_0 \cap \MP(f) \subset \partial D_f$, there exists an $N \ge 1$  depending only on $\epsilon$, $V_0$ and $f$, such that $\diam_{\EC}(V_n)<\epsilon$ for all $n\geq N$, where $V_n$ is any connected component of $f^{-n}(V_0)$.
\end{lem}

In particular, Lemma \ref{shrinking lemma} holds when $\overline{V}_0\cap \MP(f)=\emptyset$ and $V_0$ is not contained in any rotation domain, while Lemma \ref{contraction} can handle the case that $V_0$ is adjacent to the Siegel disk of $f$.

\subsection{Homotopy lifting}

Denote by $K(P_\theta)$ and $K(P_c)$ the filled Julia sets of $P_\theta$ and $P_c$, and by $\mathring{K}(P_\theta)$ and $\mathring{K}(P_c)$ the interior of
$K(P_\theta)$ and $K(P_c)$ respectively. For a quadratic rational map $G \in R_\theta^{geom}$, let $J(G)$ represent the Julia set of $G$, and let $D_G$ denote the Siegel disk of $G$.

First, suppose that $P_c$ has a super-attracting cycle $x_1$, $\cdots$, $x_p$ with period $p\ge 1$. For any $1\le i\le p$, we can take a topological disk $D_i$ containing $x_i$ such that:
\begin{enumerate}
    \item[(i)] $\overline{D_i}\cap \overline{D_j}=\emptyset$ for $1\le i\not=j\le p$ and all $\partial D_i$ are real analytic curves;
    \item[(ii)] for each $D_i$, there is an annulus $A_i$ surrounding it with $\partial D_i$ being the inner component of $\partial A_i$ and $\overline{A_i}\cap \mathcal{P}(P_c)=\emptyset$;
    \item[(iii)] $P_c$ is holomorphic in $\overline{D_i}\cup A_i$ and maps $\overline{D_i}\cup A_i$ into some $D_j$.
\end{enumerate}
These $D_i$ are called holomorphic disks of $P_c$ corresponding to the Jordan disks in condition (c) of Lemma \ref{thurston-equi} according to \cite{Zha22}.

By the definition of combinatorial equivalence in Lemma \ref{thurston-equi}, we can construct a homotopy $H_1$ such that:
\begin{enumerate}[label =(\arabic*)]
\item $H_1:S^2\times[0,1]\to \EC$ is continuous;
\item $H_1(\cdot,s)$ is a homeomorphism for any $s\in [0,1]$;
\item $H_1(\cdot,0)=\phi_1$, $H_1(\cdot,1)=\phi_{2}$;
\item $H_1(x,s)=\phi_1(x)$ for $x\in \overline{D_{P_\theta}} \cup \cup_{i=1}^p\overline{D_i}$, $s\in [0,1]$.
\end{enumerate}

Next, we define $\phi_n$ and $H_{n}$ inductively. Suppose we have $\phi_{n-1}$, $\phi_{n}$ and $H_{n-1}$ satisfying the same properties as above. Consider $F$ and $G$ as branched coverings, we can lift $H_{n-1}$ by $F$ and $G$ to obtain $H_{n}:S^2\times [0,1]\to \EC$ such that: $H_{n}$ is continuous; $H_{n}(\cdot,0)=\phi_{n}$; and 
\begin{equation}\label{commutate}
    H_{n-1}(F(x),s)=G(H_{n}(x,s)).
\end{equation}
Then define $\phi_{n+1}:=H_{n}(\cdot,1)$. It is easy to check the following properties:
\begin{enumerate}
        \item[(1)] $H_n:S^2\times[0,1]\to \EC$ is continuous;
        \item[(2)] $H_n(\cdot,s)$ is a homeomorphism for any $s\in [0,1]$;
        \item[(3)] $H_n(\cdot,0)=\phi_n$, $H_n(\cdot,1)=\phi_{n+1}$;
        \item[(4)] $H_n(x,s)=\phi_n(x)$ for $x\in F^{-(n-1)}(\overline{D_{P_\theta}}\cup\cup^p_{i=1}\overline{D_i})$, $s\in [0,1]$.
\end{enumerate} 

Note that $\phi_1(\overline{D_i})$ and $\phi_1(\overline{D_j})$ are mutually disjoint closed Jordan disks because $\overline{D_i}\cap \overline{D_j}=\emptyset$ and $\phi_1$ is a homeomorphism. Moreover, it is clear that $\phi_1(\overline{D_i})\cap\overline{D_G}=\emptyset$ for any $ 1\le i,j\le p$.

If the critical point $x_0$ is strictly preperiodic, then, as mentioned earlier, $J(P_c)$ is a dendrite in this case. By Lemma \ref{thurston-equi} and the homotopy lifting property, we obtain $\phi_{n}$ and $\phi_{n+1}$ for $n\ge 1$ which satisfy the following properties:   
\begin{enumerate}
    \item $\phi_n:S^2\to \EC$ is an orientation preserving homeomorphism;
    \item $\phi_n\circ F= G\circ \phi_{n+1}$;
    \item $\phi_n=\phi_{n+1}$ on $F^{-(n-1)}(\overline{D_{P_\theta}}\cup\mathcal{P}(P_c))$, $\phi_n(\mathcal{P}(F))=\mathcal{P}(G)$ and $\phi_n$ is conformal in $F^{-(n-1)}(D_{P_\theta})$.
\end{enumerate}

\subsection{A decomposition of the convergence region of $\phi_n$}

\begin{defi}[Drop-chain]
For the quadratic polynomial $P_\theta$, we denote the Siegel disk $U_0:=D_{P_\theta}$ as the $0$-\emph{drop} of $P_\theta$. The unique $1$-drop is defined as $U_1:=P^{-1}_{\theta}(D_{P_\theta})\setminus D_{P_\theta}$, which is a Jordan domain attached at the critical point $x_\theta:=- e^{2\pi\ii \theta}/2$. More generally, for any $n\ge 1$, each connected component $U$ of $P_{\theta}^{-(n-1)}(U_1)$ is a Jordan domain called an $n$-\emph{drop}, with $n$ representing the \textit{depth} of $U$. The map $P_{\theta}^{\circ n}:U\to D_{P_\theta}$ is conformal. 

Let $U$ and $V$ be two distinct drops of depths $m$ and $n$, respectively. Then $\overline{U}\cap \overline{V}$ is either empty or a singleton. In the latter case, we necessarily have $m\neq n$. If $m<n$, then $\overline{U}$ and $\overline{V}$ intersect at a preimage of the critical point $x_\theta$, and we say that $U$ is the \textit{parent} of $V$. 

Consider a sequence of drops $\{U_0, U_{\iota_1}$, $U_{\iota_1 \iota_2}$, $\cdots\}$, where each $U_{\iota_1 \iota_2 \cdots \iota_k}$ is the parent of $U_{\iota_1 \iota_2 \cdots \iota_{k+1}}$. The closure of the union of this sequence is defined as 
\begin{equation}
\MC=\bigcup\limits_k \overline{U_{\iota_1 \iota_2 \cdots \iota_k}}.
\end{equation}
This set $\MC$ is termed a \emph{drop-chain}.
\end{defi}

According to \cite{Pet96} and the above definitions, every point in $K(P_{\theta})$ either belongs to the closure of a drop or is the landing point of a unique drop-chain. Furthermore, different drop-chains land at distinct points. Additionally, every external ray $R_t(P_\theta)$ with $t\in\Q$ lands at a point $x\in J(P_\theta)$, which is also the landing point of a unique drop-chain, since every point on the boundary of bounded Fatou components (drops) of $P_\theta$ is not eventually periodic.  

\begin{prop}\label{graph}
For the formal mating $ F = P_\theta \sqcup P_c $, there exists a closed connected set $ T $ consisting of the closure $ \overline{D_{P_\theta}} $, finitely many drop chains of $ P_\theta $, and joined external rays of $F$, along with finitely many holomorphic disks and segments in the Fatou components of $ P_c $ if $ P_c $ has a super-attracting cycle, such that
\begin{itemize}
    \item the restriction $ \phi_n|_T $ converges uniformly to a limit map $ \phi|_T $;
    \item the complement $ \EC \setminus T $ contains finitely many connected components $ U $, where the spherical diameter of each preimage $ G^{-n}(\phi_m(U)) $ converges to zero as $ n $ goes to infinity for a fixed integer $ m $. 
\end{itemize}
\end{prop}

\begin{proof}
For a postcritically finite polynomial $P$ of degree at least two, the Hubbard tree $H\subset K(P)$ is the minimal tree that contains all points of $ \MP(P) $ such that $ P(H) \subset H $ (see \cite[Chapter 4]{DH84}).

We first assume that $ P_c $ has a super-attracting cycle $x_0=x_p, x_1, \ldots, x_{p-1} $ with period $ p > 1 $ (the case $p = 1$ is trivial). Let $\nu(i)$ denote the number of branches of the tree $H$ at the point $x_i$. According to \cite[Proposition 4.4]{DH84}, if $p > 1$, there exists a $2\leq r \leq p$ such that $\nu(i) = 1$ for $1\leq i \leq r $ and $\nu(i)=2$ for $r<i\leq p$.

In the following, we will divide the proof into five steps.

\textbf{Step 1.} Consider a vertex $x_i$ of $H$ with $\nu(i)=1$. By Lemma \ref{thurston-equi}, there exists a holomorphic disk $D_i$ containing $x_i$ such that $\phi_1|_{D_i}=\phi_2|_{D_i}$.
Let $\Omega_{x_i}$ be the Fatou component of $P_c$ containing $x_i$ and denote  $x_i^1=\partial{\Omega_{x_i}}\cap H$. 
The point $x_i^1$ is a repelling periodic point and there are two external rays of $P_c$ landing at $x_i^1$. Consequently, there are two jointed external rays of $F$ landing at $x_i^1$, which we denote as $R_{t^i_1}$ and $R_{t^i_2}$. Since $x_i^1$ is a periodic point, the angles $t^i_1$ and $t^i_2$ are rational angles. Then each jointed external ray $R_{t^i_1}$ (resp. $R_{t^i_2}$) corresponds to a unique drop-chain $\MC_{t^i_1}$ (resp. $\MC_{t^i_2}$).
At this point, we observe that the union $\overline{D_{P_\theta}} \cup R_{t^i_1} \cup R_{t^i_2}\cup \MC_{t^i_1} \cup \MC_{t^i_1}$ effectively separates $x_i$ from the other postcritical points.
Let $\lfloor x_i^1,D_i\rfloor$ be the segment connecting $x_i^1$ with $D_i$ that lies within $\Omega_{x_i}\cap H$. Define $U^i$ as the component of 
$$\EC\setminus( \overline{D_{P_\theta}}\cup R_{t^i_1} \cup R_{t^i_2}\cup \MC_{t^i_1} \cup \MC_{t^i_2}\cup \overline{D_i} \cup \lfloor x_i^1,D_i\rfloor)$$
with $\partial D_i\subset \partial U^i $. Then $U^i$ is a simply connected domain such that $\MP(F)\cap U^i=\emptyset$. 

\textbf{Step 2.} Consider a vertex $x_j$ of $H$ with $\nu(j)=2$, there exists a holomorphic disk $D_j$ containing $x_j$ such that $\phi_1|_{D_j}=\phi_2|_{D_j}$.
Again, let $\Omega_{x_j}$ denote the Fatou component of $P_c$ that contains $x_j$. In this case, $\Omega_{x_j}$ intersects $H$ at two points, which we denote as $x_j^1$ and $x_j^2$. Both $x_j^1$ and $x_j^2$ are periodic points, with two external rays of $P_c$ landing at each. These external rays are contained in four jointed external rays. We denote them as $R_{t^j_{11}}$ and $R_{t^j_{12}}$(for $x_j^1$) and $R_{t^j_{21}}$ and $R_{t^j_{22}}$(for $x_j^2$). Each jointed external ray $R_{t^j_{\iota_1\iota_2}}$ corresponds to a unique drop-chain $\MC_{t^j_{\iota_1\iota_2}}$, given that $x_j^{\iota_1}$ is periodic, where $\iota_1,\iota_2\in\{1,2\}$.
Let $\lfloor x_j^1,D_j\rfloor$ (resp. $\lfloor x_j^2,D_j\rfloor$) be the segment connected $x_j^1$ (resp. $x_j^2$) with $D_j$, contained within $\Omega_{x_j}\cap H$.
Define $U^{j1}$ (resp. $U^{j2}$) as the component of 
$$\EC\setminus(\bigcup_{\iota_1,\iota_2=1}^2(R_{t^j_{\iota_1\iota_2}}\cup \MC_{t^j_{\iota_1\iota_2}})\cup\overline{D_{P_\theta}}\cup \overline{D_j} \cup \lfloor x_j^1,D_j\rfloor\cup\lfloor x_j^2,D_j\rfloor)$$
such that $R_{t^j_{11}}$ (resp. $R_{t^j_{21}}$) and part of $\partial{D_j}$ are contained in $\partial{U^{j1}}$ (resp. $\partial{U^{j2}}$).

\textbf{Step 3.} Let 
\begin{center}
$T:=\overline{D_{P_\theta}}\cup \bigcup_{i}(\bigcup_{\iota_1=1}^2(R_{t^i_{\iota_1}}\cup \MC_{t^i_{\iota_1}})\cup \overline{D_i} \cup \lfloor x_i^1,D_i\rfloor)\cup$ \\
$\bigcup_{j}(\bigcup_{\iota_1,\iota_2=1}^2(R_{t^j_{\iota_1\iota_2}}\cup \MC_{t^j_{\iota_1\iota_2}})\cup \overline{D_j} \cup \lfloor x_j^1,D_j\rfloor\cup\lfloor x_j^2,D_j\rfloor).$
\end{center}
Then $T$ is a closed connected set. The complement 
$$\EC\setminus (T\cup \cup_i U^i \cup \cup_j (U^{j1}\cup U^{j2}))$$ 
has finitely many connected components, which we denoted by $\tilde{U}^k$. An example is illustrated in Figure \ref{T}.

\begin{figure}
    \centering
    \includegraphics[width=0.8\linewidth]{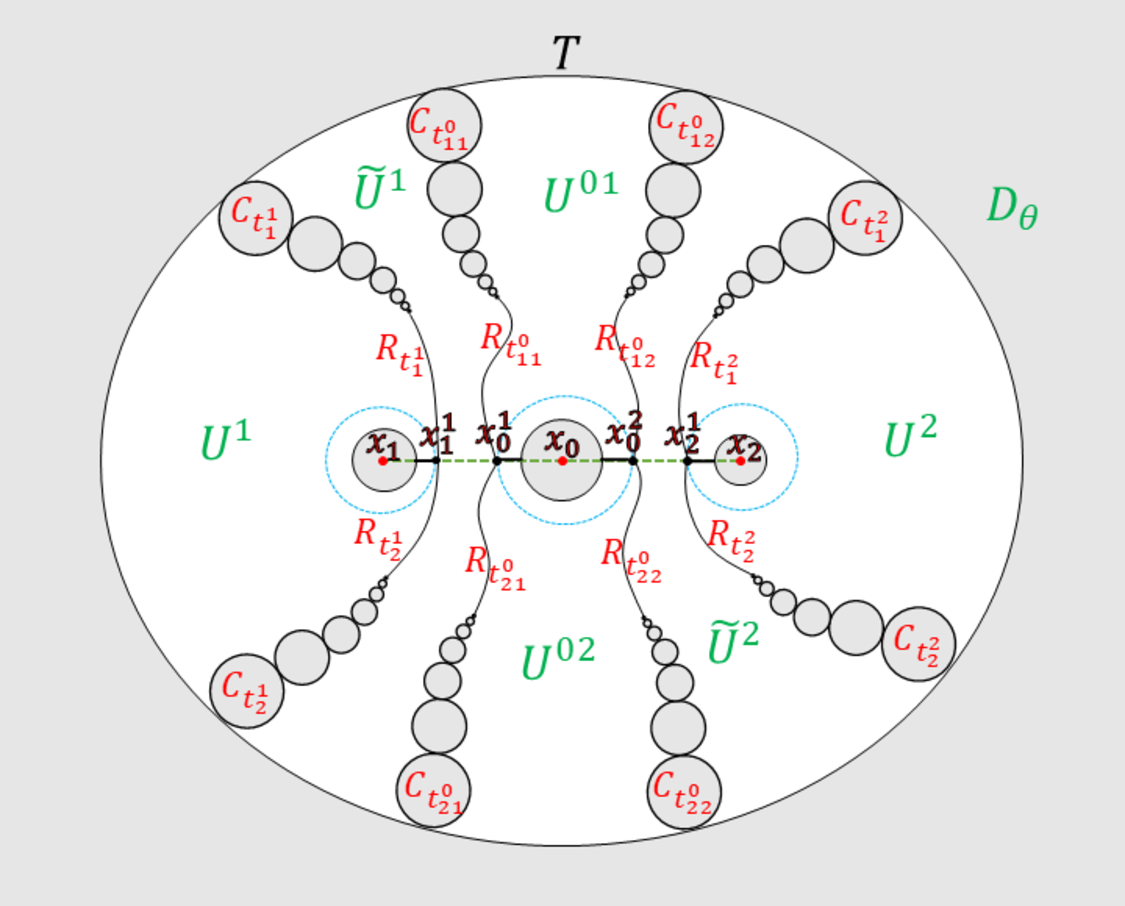}
    \caption{The construction of $T$ for $F=P_{\theta}\sqcup P_c$ with $c=-1.75488\cdots$ such that the critical point $x_0=0$ is periodic with period $3$, and $x_1=P_c(0)$, $x_2=P_c^{2}(0)$. We place the center of the Siegel disk of $F$ at infinity. The set $T$ consists of the gray regions and the black boundaries. The constructions of $U^1$ and $U^2$ for $x_1$ and $x_2$ follow Step 1, while the constructions of $U^{01}$ and $U^{02}$ for $x_0$ follow Step 2. The remaining parts are two Jordan domains $\tilde{U}^1$ and $\tilde{U}^2$.}
    \label{T}
\end{figure}

We claim that for any joined external ray $R_t$ of $F$, 
\begin{equation}\label{step_3}
    {\rm diam_{\EC}}(\phi_n(R_t)) \to 0  \hbox{  as  } n\to\infty.
\end{equation}
In fact, by Lemma \ref{thurston-equi} and the homotopy lifting property the following diagram is commutative:
\begin{equation*}
    \begin{tikzcd}
    R_t  \ar[r,"F^n"] \ar[d,"\phi_{n+1}"] & R_{2^n t} \ar[d,"\phi_1"]\\ 
    \phi_{n+1}(R_t)  \ar[r,"G^n"]   & \phi_1(R_{2^n t}).
    \end{tikzcd}
\end{equation*}
Here, $\phi_{n+1}(R_t)$ is one of the preimage components obtained by pulling back $\phi_1(R_{2^n t})$ under $G^{n}$. We know if an external ray in a ray-equivalence class is periodic (as a set), then all other rays in this class are also periodic with the same period. Then the critical point of $P_\theta$ and the critical point of $P_c$ (denoted by $x_\theta$ and $x_0$ respectively) cannot be in the same ray-equivalence class since the external rays landing at $x_\theta$ is not eventually periodic but the external rays landing at $x_0$ is eventually periodic. Furthermore, a point in $\mathcal{P}(P_c)$ and any $x\in\partial D_{P_\theta}$ cannot be in the same ray-equivalence class for the same reason. Therefore, the landing points of $\phi_1(R_{2^n t})$ cannot be contained in $\partial D_{G}$ and $\phi_1(\mathcal{P}(P_c))$ at the same time. 
By the construction of $T$, there are finitely many Jordan domains $V_0^i$'s satisfying the conditions in Lemmas \ref{shrinking lemma} and \ref{contraction}
such that $\phi_1(R_{2^n t})$ is contained in some $\overline{V_0^i}$ for any $t\in \R/\Z$. By Lemmas \ref{shrinking lemma} and \ref{contraction}, the conclusion (\ref{step_3}) holds.

\textbf{Step 4.} For each $\MC_t \cup R_t \subset T$, where $$\MC_t=\bigcup_k \overline{U^t_{\iota_1 \iota_2 \cdots \iota_k}},$$
let $x_t'$, $x_t''$ be the two landing points of $R_t$ such that $\overline{\MC_t}\cap R_t=\{x_t'\}$ and $x_t''$ is equal to one of $x_i^1$, $x_j^1$ or $x_j^2$ in Steps 1 and 2. We claim that there exists a periodic point $z_t$ such that for any $x\in R_t$, $\phi_n(x)$ converges to $z_t$. 

Let $\Omega_i$ be the fixed Fatou component of $F^q$ such that $x_i\in D_i\subset \Omega_i$ and $x_t''\in \partial \Omega_i$. Replacing $F^q$ by $\hat{F}$, the Fatou component $\Omega_i$ becomes a fixed super-attracting domain of $\hat{F}$. We can then take a sequence $\{y_n\}_{n=1}^{\infty}$ such that:
\begin{enumerate}
    \item $y_n\in \Omega_i$ for any $n\ge 1$,
    \item $\lim_{n\to \infty} y_n=x_t''$,
    \item $y_1\in D_i$ and $\hat{F}(y_{n+1})=y_n$ for any $n\ge 1$.
\end{enumerate}
Replacing the sequence $\{\phi_n\}_{n\ge 1}$ with $\{\phi_{(n-1)p+1}\}_{n\ge 1}$. For simplicity, we will continue to denote the family of maps $\{\phi_{(n-1)p+1}\}_{n\ge 1}$ as $\{\phi_n\}_{n\ge 1}$. From (c), we know that $\phi_{n}(y_n)=\phi_{n+1}(y_n)$ for any $n\ge 1$, where $\phi_n\circ F^q=G^q\circ \phi_{n+1}$. Enclosing $\{y_n\}_{n=1}^{\infty}$ and $x_t''$ within a Jordan domain $U_1$ such that Lemma \ref{shrinking lemma} can be applied to $V_1=\phi_1(U_1)$, we conclude that for any $\epsilon>0$, there exists an $N_1 \ge 1$ such that $\diam_{\EC}(V_n)<\epsilon$ for all $n\ge N$, where $V_n$ is any component of $\hat{G}^{-n}(V_1)$. Since $x_t''$ is a repelling fixed point of $\hat{F}$ and $\hat{F}(y_{n+1})=y_n$, we can choose $V_n$ such that $\phi_n(y_n)\in V_n$ and $\phi_n(x_t'')\in V_n$. Therefore, for any $m,n$ with $m>n\ge N_1$, we have:
\begin{equation*}
\begin{split}
    &\dist_{\EC}(\phi_n(x_t''),\phi_m(x_t''))\\
    \le~ &\dist_{\EC}(\phi_n(x_t''),\phi_n(y_n))+\dist_{\EC}(\phi_n(y_n),\phi_m(y_n))\\
    &+\dist_{\EC}(\phi_m(y_n),\phi_m(y_m))+\dist_{\EC}(\phi_m(y_m),\phi_m(x_t''))\\
    \le~&\diam_{\EC}(V_n)+0+\dist_{\EC}(\phi_m(y_n),\phi_m(y_m))+\diam_{\EC}(V_m).
\end{split}
\end{equation*}
Since for any $y\in\Omega_i$, $\phi_n(y)$ converges as $n\to \infty$, the sequence $\{\phi_n|_{\Omega_i}\}_{n\ge 1}$ forms a normal family. Meanwhile, since $\{y_n\}_{n=1}^{\infty}$ is a Cauchy sequence, there exists $N_2$ such that for all $m>n\ge N_2$, we have
$$\dist_{\EC}(\phi_m(y_n),\phi_m(y_m))< \epsilon$$ 
by the Arzel\`{a}–Ascoli theorem (see \cite{Ah79}). Thus,
\begin{equation*}
    \dist_{\EC}(\phi_n(x_t''),\phi_m(x_t''))<3\epsilon.
\end{equation*}
Therefore, $\phi_n(x_t'')$ converges to a point $z_t$. Combined with (\ref{step_3}), we conclude that $\phi_n(x)$ converges to a point $z_t$ for any $x\in R_t$. Since $R_t$ is periodic, it follows that $z_t$ is a periodic point of $G$ with period $q'|q$. By replacing $G^q$ with $\hat{G}$ for simplicity of notation, the point $z_t$ becomes a repelling fixed point of $\hat{G}$.

For each $\MC_t \cup R_t \subset T$, we can identify a Jordan domain $V_0$ such that $\phi_1(\MC_t \cup R_t) \subset \overline{V_0}$, $z_t \in V_0$, and $V_0$ satisfies the conditions outlined in Lemma \ref{contraction}. By Lemma \ref{contraction}, for any $\epsilon > 0$, there exists $N \in \mathbb{N}^+$ such that for all $n \geq N$, the spherical diameter of $V_n$ is less than $\epsilon$, where $V_n$ is any component of $\hat{G}^{-n}(V_0)$. Let $\hat{V}_n$ denote the component of $\hat{G}^{-n}(V_0)$ that contains $z_t$. We take $\epsilon$ small enough such that $\hat{V}_n$ is contained in the linearization domain of the repelling fixed point $z_t$ for all $n\ge N$.

For any $s \in \mathbb{N}^+$ and $n \geq N$, if $x \in \bigcup\limits_{k=1}^{N} U_{\iota_1 \iota_2 \cdots \iota_k}^t$, we have $\phi_{n+1}(x) = \phi_{n+2}(x)$ for any $n\ge N$. If 
$$x \in R_t \cup \big(\MC_t \setminus \bigcup\limits_{k=1}^{N} U_{\iota_1 \iota_2 \cdots \iota_k}^t\big),$$ 
then $\phi_N(x)$ and $\phi_{N+s}(x)$ are contained a disk centered at $z_t$ with radius $\epsilon$. Thus, the following inequality holds: 
\[
{\rm dist}_{\EC}(\phi_N(x), \phi_{N+s}(x)) \leq {\rm dist}_{\EC}(\phi_N(x), z_t) + {\rm dist}_{\EC}(z_t, \phi_{N+s}(x)) \leq 2\epsilon.
\]
By the choice of $\epsilon$, there exists a constant $C_0\ge 1$ such that
\begin{equation}
    {\rm dist}_{\EC}(\phi_n(x), \phi_{n+s}(x))\le C_0\cdot{\rm dist}_{\EC}(\phi_N(x), \phi_{N+s}(x))
\end{equation}
for any $n\ge N$. Therefore, $\phi_n|_{\MC_t \cup R_t}$ converges uniformly as $n \to \infty$.

Meanwhile, the segment $\lfloor x_i^1, D_i \rfloor$ (resp. $\lfloor x_j^{\iota_1}, D_j \rfloor$) $\subset H$ lies within the Fatou set of $P_c$. The spherical diameter of the segment $\lfloor x_i^1, D_i \rfloor$ (resp. $\lfloor x_j^{\iota_1}, D_j \rfloor$, where $\iota_1 = 1, 2$) over which $\phi_n(x) \neq \phi_{n+1}(x)$ converges to $0$ as $n \to \infty$. Using the same reasoning, this convergence is uniform.

Furthermore, since $T$ is a finite graph, $\phi_n|_T$ converges uniformly. We denote this uniform convergence as $\phi_n|_T \to \phi|_T$ as $n \to \infty$.

\textbf{Step 5.} Let $m$ be a fixed integer. Each $\tilde{V}^k_m=\phi_m(\tilde{U}^k)$ is a Jordan domain with $\tilde{V}^k_m\cap \MP(G)=\emptyset$ and $\partial{\tilde{V}^k_m}\cap \MP(G)\subset \partial{D_{G}}$. By Lemma \ref{contraction} the spherical diameter of any connected component of $G^{-n}(\tilde{V}^k_m)$ converges to zero as $n$ goes to infinity. For each $V_m^i=\phi_m(U^i)$ (resp. $V_m^{j\iota_1}=\phi_m(U^{j\iota_1})$ with $\iota_1=1,2$) obtained in Step 1 (resp. Step 2), it is a simply connected domain (but not a Jordan domain) with no postcritical points in its interior. For simplification, we omit the superscript. By the Riemann-Hurwitz Formula, we have $\deg(G^n:V_{m+n}\to V_m)=1$, where $V_{m+n}$ is any connected component of $G^{-n}(V_m)$ and $V_m=V_m^i$ or $V_m^{j\iota_1}$. Moreover, since $V_m$ is contained within the union of finitely many Jordan domains satisfying the conditions in Lemma \ref{contraction}, it follows that $\diam_{\EC}(V_{m+n})\to 0$ as $n\to\infty$. 

It remains to consider the case where the critical point $x_0$ of $P_c$ is strictly preperiodic. The proof in this case is similar to that of the periodic case, but in a simpler way; we will highlight the differences at each step.

\begin{figure}[htbp]
  \setlength{\unitlength}{1mm}
  \subfigure{\includegraphics[width=.49\textwidth]{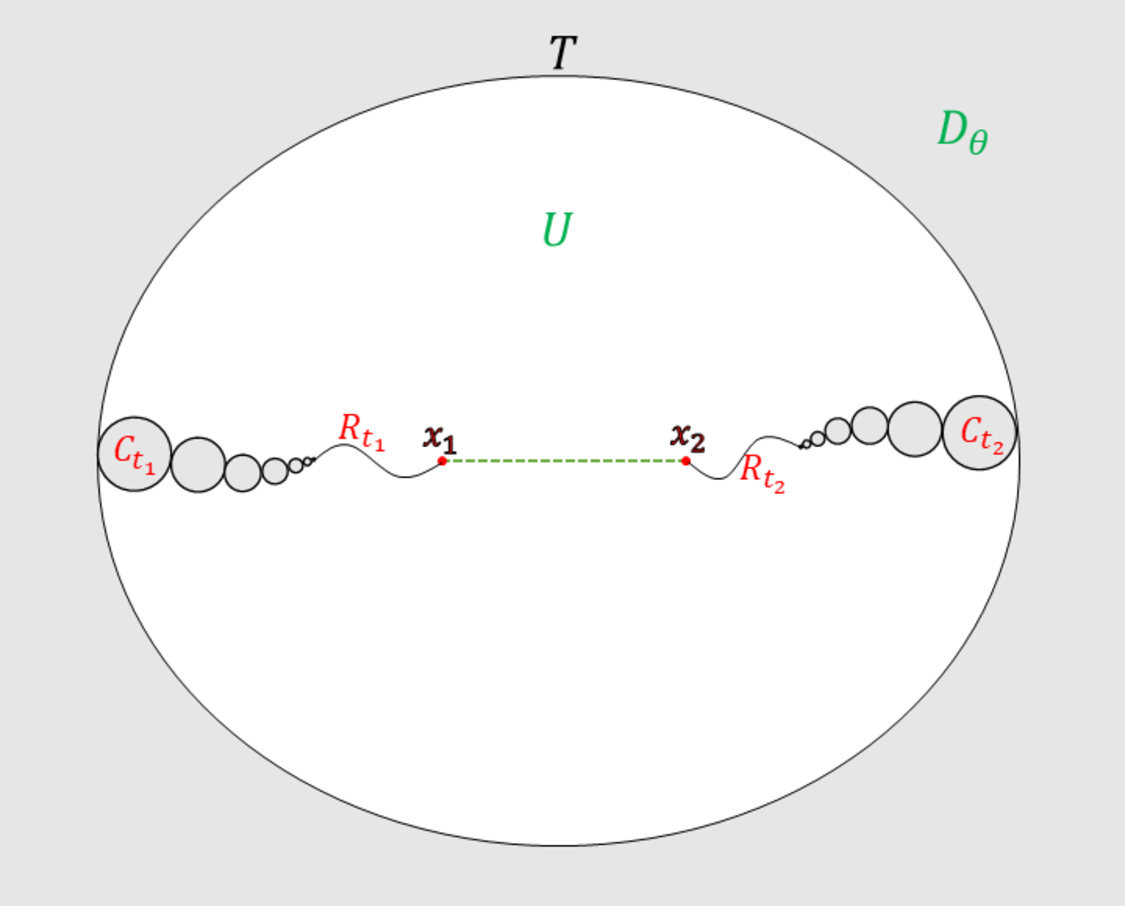}}
  \subfigure{\includegraphics[width=.49\textwidth]{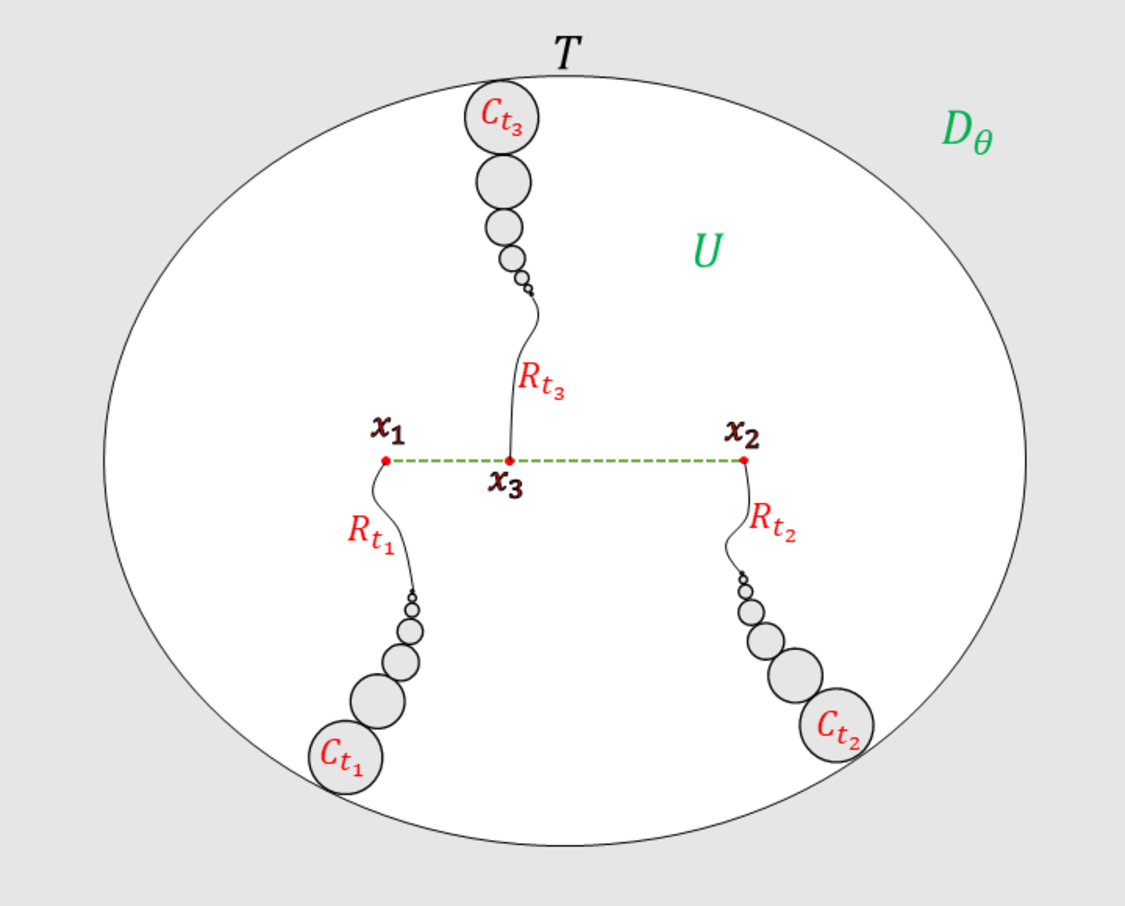}}
  \caption{The picture on the left shows the construction of $T$ for $F=P_{\theta}\sqcup P_c$ with $c=-2$. Here, $x_2$ is a fixed point with only one joined external ray $R_{t_2}$ from $F$ landing at it, while $x_1$ is another preimage of $x_2$ under $P_c$. The picture on the right displays the construction of $T$ for $F=P_{\theta}\sqcup P_c$ with $c=-1.54368\cdots$ such that $x_1=P_c(0)$, $x_2=P^2_c(0)$, and $x_3=P^3_c(0)$ is a fixed points of $P_c$. There are multiple external rays of $P_c$ landing at each $x_i$, from which we can select one to construct the set $T$.}
  \label{T_12}
\end{figure}

Differences in Step 1 and Step 2: For a vertex $x_i$ of $H$, we have $x_i\in J(P_c)$ such that at least one external ray lands at $x_i$. We denote one of these rays as $R_{t_i}(P_c)$. The corresponding joined external ray and drop-chain are $R_{t_i}$ and $\MC_{t_i}$.

Differences in Step 3: Let
\begin{center}
$T=\overline{D_{P_\theta}}\cup (\bigcup_{i}(R_{t_i} \cup \MC_{t_i})).$
\end{center}
See Figure \ref{T_12} for two examples. 

Differences in step 4: For each $\MC_t \cup R_t \subset T$, we can identify a Jordan domain $V_0$ such that $\phi_1(\MC_t \cup R_t\setminus \overline{U_{\iota_1}^t}) \subset \overline{V_0}$ and $V_0$ satisfies the conditions outlined in Lemma \ref{shrinking lemma}. 

Differences in step 5: The complement $U$ of $\EC\setminus T$ is a simply connected domain. While it has postcritical points on its boundary, but contains no postcritical points in its interior. Furthermore, $V_m=\phi_m(U)$ is contained within the union of finitely many Jordan domains that satisfy the conditions in either Lemma \ref{shrinking lemma} or \ref{contraction}.

Based on the above modifications, we can prove $\phi_n|_{T}$ converges uniformly to $\phi|_{T}$ and $\diam_{\EC}(V_{m+n})\to 0$ as $n\to\infty$ in the strictly preperiodic case. This concludes the proof.
\end{proof}

\subsection{Convergence of homotopy lifting}
\begin{lem}\label{siegel-semi-conjugate}
    There exists a continuous map $\phi: S^2\to \EC$ satisfying 
    \begin{enumerate}
        \item $\phi\circ F= G \circ \phi$;
        \item $\phi$ is a uniform limit of orientation preserving homeomorphisms;
        \item $\phi$ is conformal in $\mathring{K}(P_\theta)\sqcup\mathring{K}(P_c)$ onto $\EC\setminus J(G)$ and $\phi^{-1}(\EC\setminus J(G))$ $=\mathring{K}(P_\theta)\sqcup\mathring{K}(P_c) $.
    \end{enumerate}
\end{lem}

\begin{proof}
By Proposition \ref{graph}, $\phi_n(x)$ converges uniformly to $\phi(x)$ for $x\in T$. Furthermore, from the construction of $T$, we know that for any $\epsilon>0$, there exist $N\in \N^+$ and finitely many small disks $D^{\epsilon}_1,\cdots,D^{\epsilon}_r$ with $\diam_{\EC}(D^{\epsilon}_i)\le \epsilon$, such that for $m\ge N$ and any $s\in\N^+$, we have $\phi_m(x)=\phi_{m+s}(x)$ for $\phi_m(x)\in \phi_m(T)\setminus (\cup_{i=1}^r D^{\epsilon}_i)$. For any $s\in\N^+$, we consider the arc
$$\alpha_{m,s,x}=H_m(x,\cdot)\cup H_{m+1}(x,\cdot)\cup\cdots\cup H_{m+s-1}(x,\cdot):[0,1]\to \EC$$ connecting $\phi_m(x)$ with $\phi_{m+s}(x)$ for some fixed $x\in \EC$.
Since $\phi_n(x)$ converges uniformly to $\phi(x)$ for $x\in T$, we can choose $D^{\epsilon}_i$ so that for any $x\in T\cap D^{\epsilon}_i$,  we have $\alpha_{m,s,x}\in D^{\epsilon}_i$ for all $m\ge N$ and $s\in \N^+$, where $N$ is sufficiently large.

Let $x\notin T$. Then $x$ lies in some component $U$ of $\EC\setminus T$. There are two cases to consider:

\textbf{Case 1:} $\alpha_{m,s,x}\subset\phi_m(U)$. By Proposition \ref{graph}, the spherical diameter of any preimage component of $\phi_m(U)$ under $G^n$ converges to zero as $n\to \infty$. Therefore, the lifting of $\alpha_{m,s,x}$ is contained within these components. Thus, there exists $N'\in \N^+$, such that $\dist_{\EC}(\phi_{n}(x),\phi_{n+s}(x))\le \epsilon$ for any $n\ge N'$ and any $s\in \N^+$. This means that $\phi_n(x)$ converges uniformly to $\phi(x)$ for $x\in U$. 

\textbf{Case 2:} $\alpha_{m,s,x}\cap \partial{\phi_m(U)}\not=\emptyset$. Suppose $z'\in\alpha_{m,s,x}\cap \partial{\phi_m(U)}$. If $z'\in\phi_m(T)\setminus(\cup_{i=1}^r D^{\epsilon}_i)$, then $\phi_m(x)=\phi_{m+s}(x)$ for any $s\in \N^+$, which implies that $\phi_m(x)$ converges uniformly to $\phi(x)$ as $m\to \infty$. If $z'\in\cup_{i=1}^r D^{\epsilon}_i$, $\dist_{\EC}(\phi_m(x),\phi_{m+s}(x))<\epsilon$ for any $s\in \N^+$ by Lemma \ref{shrinking lemma}. This also means that $\phi_m(x)$ converges uniformly to $\phi(x)$ as $m\to \infty$. 

Thus, we have proven that $\phi_m(x)$ converges uniformly to $\phi(x)$ for any $x\in S^2$.
It is easy to verify that the limit map $\phi$ satisfies the conditions (a), (b) and (c).
\end{proof}

\subsection{Properties of the semi-conjugacy $\phi$}

\begin{lem}
    $\phi$ is surjective.
\end{lem}

\begin{proof}
Since $\phi$ is a homeomorphism in the interior of $K(P_\theta )\sqcup K(P_c)$ and 
$$\EC\setminus J(G)=\cup_{n\ge 0}\phi_{n+1}( F^{-n}(D_{P_\theta}\cup \cup_i D_i)),$$ 
every $z\in \EC\setminus J(G)$ has a unique preimage under $\phi$. It remains to consider every $z\in J(G)$. Since $\phi:\partial D_{P_\theta} \to \partial D_G\subset J(G)$, for a $z_0\in \partial D_G$, $\cup_{n\ge 0}\{G^{-n}(z_0)\}$ is everywhere dense in $J(G)$ by \cite[Corollary 4.13]{Mil06}. Together with that $\phi$ is continuous, we conclude that $\phi$ is surjective.
\end{proof}

\begin{lem}\label{property}
For any $z\in\EC$, $\phi^{-1}(z)$ is closed, connected, and nonseparating.
\end{lem}

\begin{proof}
The proof is similar to the Proposition 3.1(i) in \cite{Shi00}. Since $\phi$ is a continuous map,  the closeness of $\phi^{-1}(z)$ is obvious.

Suppose that $\phi^{-1}(z)$ is disconnected for some $z\in\EC$. Then there are two disjoint open sets $\Omega_1$, $\Omega_2\subset S^2$ such that 
\begin{center}
    $\phi^{-1}(z)\subset \Omega_1\cup \Omega_2$, $\Omega_1\cap \phi^{-1}(z)\not=\emptyset$ and $\Omega_2\cap \phi^{-1}(z)\not=\emptyset$. 
\end{center}
Clearly $\phi(S^2\setminus(\Omega_1\cup \Omega_2) )$ is a closed set and $z\notin \phi(S^2\setminus(\Omega_1\cup \Omega_2) )$. Let $O(z)$ be a neighborhood of $z$ which is small enough such that 
$$\phi(S^2\setminus(\Omega_1\cup \Omega_2))\cap \overline{O(z)} =\emptyset.$$ 
Since $\phi$ is a uniform limit of a sequence of homeomorphisms $\{\phi_n\}_{n\ge 1}$, there exists a homeomorphism $\phi_n:S^2\to \EC$ such that 
$$\phi_n(S^2\setminus(\Omega_1\cup \Omega_2))\cap\overline{O(z)} =\emptyset.$$
Then 
\begin{center}
    $\phi^{-1}_n(O(z))\subset \Omega_1\cup \Omega_2$ and 
$\phi_n(\phi^{-1}(z))\subset O(z)$. 
\end{center}
Hence 
\begin{center}
    $\phi^{-1}_n(O(z))\cap\Omega_1\not=\emptyset$ and
$\phi^{-1}_n(O(z))\cap\Omega_2\not=\emptyset$. 
\end{center}
This contradicts the fact that $\phi_n$ is a homeomorphism and $\phi^{-1}_n(O(z))$ is connected.

Suppose that $\phi^{-1}(z)$ separates $S^2$. Then there exist two nonempty sets $\Omega_1$, $\Omega_2$ in $S^2$ such that $\Omega_1\cap \Omega_2=\emptyset$ and $S^2\setminus \phi^{-1}(z)= \Omega_1\cup \Omega_2$. For any $w\in \EC\setminus\{z\}$, $\phi^{-1}(w)\in\Omega_1 $ or $\phi^{-1}(w)\in\Omega_2 $ since $\phi^{-1}(w)$ is connected. Let 
$$W_1:=\{w\in \EC\setminus\{z\}| \phi^{-1}(w)\subset\Omega_1\};$$
$$W_2:=\{w\in \EC\setminus\{z\}| \phi^{-1}(w)\subset\Omega_2\}.$$
It is clear that $W_1\cap W_2=\emptyset$ and $W_1\cup W_2=\EC\setminus\{z\}$. Meanwhile
$W_1\not=\emptyset$ and $W_2\not=\emptyset$ since $\Omega_1$, $\Omega_2$ are two non-empty sets. By the continuity of $\phi$, the set function $w\mapsto \phi^{-1}(w)$ is `upper semi-continuous'. Hence both $W_1$ and $W_2$ are open sets. This is impossible since $\EC\setminus\{z\}$ is connected.
\end{proof}

\section{Ray equivalence}
Recall that $x_\theta$ is the critical point of $P_\theta$. We have the following lemma:

\begin{lem}\label{different preimages}
    Let $x, x'\in J(P_\theta)$ be any two different preimages of $x_\theta$ under $P_\theta$. Then $\phi(x)\not=\phi(x')$.
\end{lem}

\begin{proof}
Let $x, x'\in J(P_\theta)$ and $m,n\in \N$ be two distinct integers such that $P^{n}_\theta(x)=x_\theta$ and $P^{m}_\theta(x')=x_\theta$. Without loss of generality, assume that $n>m$. Suppose that $\phi(x)=\phi(x')$. By condition (a) in Lemma \ref{siegel-semi-conjugate},
we have
$$\phi(F^n(x))=G^n(\phi(x))=G^n(\phi(x'))=\phi(F^n(x')).$$
Then
$$\phi(x_\theta)=\phi(F^{n-m}\circ F^m(x')) =\phi(F^{n-m}(x_\theta)).$$ 
Note that $F^{n-m}(x_\theta)\in \partial D_{P_\theta}$ and  
 $F^{n-m}(x_\theta)\not= x_\theta$. This implies that the boundary of $D_G=\phi(D_{P_\theta})$ must have at least one self-intersection point, and hence cannot be a Jordan curve. This contradicts  the fact that $D_{G}$ is a Siegel disk of $G$, as known in Lemma \ref{thurston-equi}.  

If $m=n$ such that $P^{n}_\theta(x)=x_\theta$ and $P^{m}_\theta(x')=x_\theta$. Without loss of generality, assume that $n=m=1$. Suppose $\phi(x)=\phi(x')$. Then $\phi(x)\in \partial D_{G}$ and there are at least four preimage components of $D_G$ attaching at $\phi(x)$. The orbit of $\phi(x)$ under $G$ will meet the critical point of $G$ once, so that one of the preimage components of $D_G$ forms a wandering orbit, which is impossible.
\end{proof}

\begin{lem}\label{finite many rays}
For any point $y$ in $J(P_c)$, suppose that at most $m\in \N^+$ external rays land at $y$. Then, in each ray-equivalence class of $F=P_\theta\sqcup P_c$, there are at most $2m$ joined external rays. 
\end{lem}

\begin{proof}
Note that for any $x\in J(P_\theta)$, at most two external rays of $P_\theta$ land at $x$. Define the sets 
\begin{equation*}
\begin{split}
        &X_1=\{x\in J(P_\theta)|\text{ there exists an } n\in \N  \text{ such that } P_\theta^n(x)=x_\theta \},\\
        &X_2=J(P_\theta)\setminus X_1.
\end{split}
\end{equation*}
There are two external rays landing at $x$ when $x\in X_1$, and only one external ray lands at $x$ when $x\in X_2$.

Denote a ray-equivalence class of $F$ by $\mathcal{R}$, and suppose $\mathcal{R}\cap X_1=\emptyset$. Then, there are at most $m$ joined external rays in $\mathcal{R}$, since every joined external rays in $\mathcal{R}$ consists of an external ray of $P_c$ landing at a common point $y\in J(P_c)$ and an external ray of $P_\theta$ landing at a point $x\in X_2$, along with their respective landing points.

Now suppose $\mathcal{R}\cap X_1\not=\emptyset$. Let $x_1\in \mathcal{R}\cap X_1$, and suppose $x_1$ is connected to $y_1\in J(P_c)$ by a joined external ray. Without loss of generality, assume there are $m$ external rays landing at $y_1$. Then, the remaining $m-1$ joined external rays connect $y_1$ to $x_2, x_3,\cdots,x_m$, respectively. 

If $x_2\in X_1$, by (\ref{step_3}), we would have
\begin{equation*}
    \phi(x_1)=\phi(y_1)=\phi(x_2),
\end{equation*}
which contradicts Lemma \ref{different preimages}. Therefore, $x_2\in X_2$. Similarly $x_3,\cdots,x_m\in X_2$. Since $x_1$ has two external rays landing at it, the other external ray is the same case. We conclude that there are at most $2m$ joined external rays in $\mathcal{R}$. 
\begin{figure}[!htpb]
 \setlength{\unitlength}{0.1mm}
  \centering
  \includegraphics[width=0.65\textwidth]{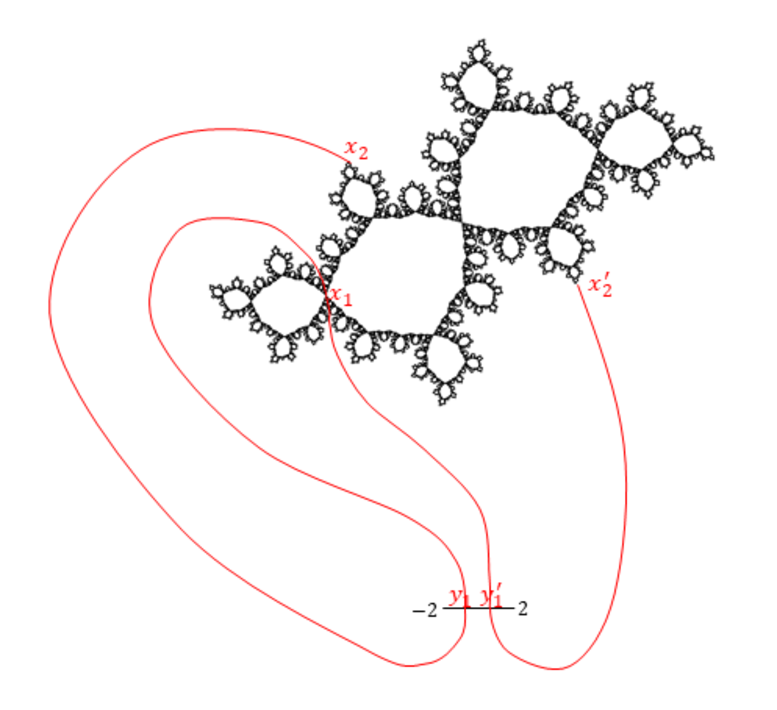}
  \caption{An illustration for a ray-equivalence class of $P_{\sqrt[3]{1/4}}\sqcup P_{-2}$.} 
  \label{joined external rays}
\end{figure}
\end{proof}

\begin{exam}[See Figure \ref{joined external rays}]
The set colored black in the upper half of Figure \ref{joined external rays} represents the Julia set of $P_{\theta}(z)=z^2+e^{2\pi i \theta}z$ with $\theta=\sqrt[3]{1/4}$, and the large region in the lower left of it is a Siegel disk (see Figure 26a in \cite{Mil06}). The closed interval in the lower half of Figure \ref{joined external rays} represents the Julia set of $P_{-2}(z)=z^2-2$. 
    
There are four joined external rays colored red, which connect the following points: $x_1$ and $y_1$; $y_1$ and $x_2$; $x_1$ and $y'_1$; $y'_1$ and $x'_2$, where $x_1\in X_1$ is a landing point of exactly two external rays, and $x_2, x'_2\in X_2$, each of which is a landing point of only one external ray. Note that for all $y\in J(P_{-2})\setminus\{-2,2\}$, two external rays land at $y$. These four joined external rays form a ray-equivalence class of $P_{\sqrt[3]{1/4}}\sqcup P_{-2}$.
    
This ray-equivalence class cannot form a loop connecting $x_1$ with $y_1$ by two different joined external rays by Lemma \ref{property}. (The positions of the red joined external rays in this figure are illustrative and do not represent the exact locations of the ray-equivalence class of $P_{\sqrt[3]{1/4}}\sqcup P_{-2}$.)
\end{exam}

For a given quadratic postcritically finite polynomial $P_c$, a point $\tilde{y}\in J(P_c)$ is a \emph{branching point} if $J(P_c)\setminus\{\tilde{y}\}$ has at least three connected components. Note that if at least three external rays land at $\tilde{y}$, then $\tilde{y}$ is a branching point. According to Thurston's result (see \cite{Thu09}), the orbit of such a branching point contains either a periodic point or a critical point. Thus, $\tilde{y}$ is eventually periodic. 
Moreover, there are only finitely many external rays landing at any periodic point of $J(P_c)$ by \cite[Lemma 18.12]{Mil06}, and $\tilde{y}$ also has finitely many external rays landing as its orbit can meet the critical point at most once. We conclude that for any postcritically finite quadratic polynomial $P_c$ and for all $y\in J(P_c)$, there are finitely many external rays landing at $y$. By Lemma \ref{finite many rays}, each ray-equivalence class of $F=P_\theta\sqcup P_c$, given a specific $P_c$, has finitely many joined external rays. 

\begin{lem}\label{ray equivalent}
For $x,y\in S^2$, $\phi(x) = \phi(y)$ if and only if $x$ and $y$ are ray equivalent. 
\end{lem}

\begin{proof}
Suppose $x,y\in S^2$ are ray equivalent. By Lemma \ref{finite many rays}, $x$ and $y$ are connected by finitely many joined external rays. These rays are mapped to a common point under $\phi$ by the continuity of $\phi$ and (\ref{step_3}). Therefore, $\phi(x) = \phi(y)$.

Now, suppose $\phi(x) = \phi(y)=z$. By Lemma \ref{finite many rays}, we know that each ray-equivalence class is closed. Hence, $\phi^{-1}(z)$ contains only one ray-equivalence class. If $\phi^{-1}(z)$ contained more than one class, it would be disconnected, which contradicts Lemma \ref{property}. Therefore, $x$ and $y$ must be ray equivalent.
\end{proof}

From now on, we know that for every $z\in J(G)$, $\phi^{-1}(z)$ is a closed ray-equivalence class, and $\phi^{-1}(z)$ is connected and nonseparating. Therefore, the following lemma can be applied:
\begin{lem}[\cite{Moo25}]
 Suppose that $\sim$ is a closed equivalence relation on the 2-sphere $S^2$ such that every equivalent class is a compact connected nonseparating proper subset of $S^2$. Then the quotient space $S^2/\sim$ is again homeomorphic to $S^2$. 
\end{lem}

As a immediate corollary we have:

\begin{cor}\label{surjective}
    $\phi(K(P_\theta))\cup \phi(K(P_c))=\EC$.
\end{cor}

The proof of the Main Theorem follows from (\ref{step_3}), Lemma \ref{siegel-semi-conjugate}, Lemma \ref{ray equivalent}, and Corollary \ref{surjective}.

\newcommand{\etalchar}[1]{$^{#1}$}
\providecommand{\bysame}{\leavevmode\hbox to3em{\hrulefill}\thinspace}
\providecommand{\MR}{\relax\ifhmode\unskip\space\fi MR }
\providecommand{\MRhref}[2]{%
  \href{http://www.ams.org/mathscinet-getitem?mr=#1}{#2}}
\providecommand{\href}[2]{#2}

\end{document}